\documentclass[12pt]{article}

\usepackage{amsmath,amssymb,amsthm,amsfonts, fullpage, setspace}

\usepackage{eucal}
\usepackage{subeqnarray,bm}
 \usepackage[utf8]{inputenc}
\usepackage{graphicx}[final]
\usepackage{psfrag} 
\usepackage{color}
\usepackage{pdfpages}

\usepackage{epstopdf}

\usepackage[colorlinks]{hyperref}

\usepackage{hyperref}
\hypersetup{colorlinks=red,citecolor=blue,pdfstartview=FitH, pdfpagemode=None}
\usepackage{multirow}
\usepackage[bottom]{footmisc}
\usepackage{ifthen} 
\usepackage{sectsty}
\usepackage{subfig}
\usepackage{float}
\usepackage[english]{babel}
\usepackage[nottoc]{tocbibind}
 \DeclareGraphicsExtensions{.eps,.png,.pdf}
\usepackage{mwe}
\makeatletter
\def\maxwidth{ %
  \ifdim\Gin@nat@width>\linewidth
    \linewidth
  \else
    \Gin@nat@width
  \fi
}
\makeatother

\definecolor{fgcolor}{rgb}{0.345, 0.345, 0.345}

\allsectionsfont{\normalsize\bfseries}

\usepackage{framed}
\makeatletter
 {\par\unskip\endMakeFramed%
 \at@end@of@kframe}
\makeatother

\definecolor{shadecolor}{rgb}{.97, .97, .97}
\definecolor{messagecolor}{rgb}{0, 0, 0}
\definecolor{warningcolor}{rgb}{1, 0, 1}
\definecolor{errorcolor}{rgb}{1, 0, 0}

\newtheorem{thm}{Theorem}[section]

\newtheorem{defn}[thm]{Definition}

\newtheorem{lemma}[thm]{Lemma}

\newtheorem{remark}[thm]{Remark}

\newtheorem*{thm*}{{\bf Main Theorem}}

\newcommand{\ds}{\displaystyle}

\newcommand{\norm}[1]{\left\Vert#1\right\Vert}
\newcommand{\abs}[1]{\left\vert#1\right\vert}
\newcommand{\set}[1]{\left\{#1\right\}}

\newcommand{\Intv}[1]{\left[#1\right]}
\newcommand{\rb}[1]{\left(#1\right)}

\newcommand{\C}{\mathbb{C}}

\newcommand{\XXint}{\int_{\mathbb{X}^2}}

\newcommand{\N}{\mathbb{N}}

\newcommand{\R}{\mathbb{R}}
\newcommand{\pB}{\pazocal{B}}
\newcommand{\pH}{\pazocal{H}}
\newcommand{\pD}{\pazocal{D}}
\newcommand{\T}{\mathbb{T}}

\def\XXint#1#2#3{{\setbox0=\hbox{$#1{#2#3}{\int}$ }
\vcenter{\hbox{$#2#3$ }}\kern-.6\wd0}}

\numberwithin{equation}{section}
\usepackage{calrsfs}
\DeclareMathAlphabet{\pazocal}{OMS}{zplm}{m}{n}

\allsectionsfont{\normalsize\bfseries}

\numberwithin{equation}{section}
\allowdisplaybreaks

\begin{document}
\title{Analytic characterization of high dimension weighted  special atom spaces}
\author{Eddy Kwessi\footnote{Corresponding author, Department of Mathematics, Trinity University, San Antonio, TX 78212, USA;
ekwessi@trinity.edu}\;\quad  and \quad \;Geraldo de Souza\footnote{Department of Mathematics, Auburn University, Auburn, AL 36849, USA.}}

\date{}
\maketitle

\begin{abstract}


Special atom spaces have been around for quite awhile since the introduction of atoms by R. Coifman in his seminal paper who led to another proof that the dual of the  Hardy space $H^1$ is in fact the space of functions of bounded means oscillations (BMO). Special atom spaces enjoy quite a few attributes of their own, among which the fact that they have   an analytic extension to  the unit disc. Recently, an extension of special atom spaces to higher dimensions was proposed, making ripe the possible exploration of the above extension in higher dimensions. In this paper we    propose an   analytic characterization of  special atom spaces  in higher dimensions.

  
\end{abstract}

\maketitle

MSC Classification: 32C20, 32C37, 32K05 32K12
\section{Introduction}


Let  $d$ be some positive  integer. We define the unit disk as 
$\mathbb{D}=\{z\in \C: \abs{z}<1\}$, the sphere  as $\T=\{z\in \C: \abs{z}=1\}$, the polydisk and polysphere respectively   as $\mathbb{D}^d$ and $\T^d$. Atoms were introduced by  Coifman in \cite{Coifman1974} as a tool to explicitly represent functions in the Hardy  space  $H^p$ for $0<p\leq 1$. The following definition was proposed:
\begin{defn}

Let $0<p\leq 1$ and an interval $J$ of $\R$. An atom is a function $b$ defined on the interval  $J$   and satisfying 
\begin{enumerate}
\item $\abs{b(\xi)}\leq \dfrac{1}{\abs{J}^{1/p}}$\;.
\item $\ds \int_{-\infty}^{\infty} \xi^kb(\xi)d\xi=0$, \quad for $0\leq k\leq \Intv{\frac{1}{p}}-1 $, where $[x]$ is the integer part of $x$\;.
\end{enumerate}
\end{defn}
\noindent From this definition, functions in $H^p(\R)$ could now be characterized via their atomic decomposition in the following theorem:
\begin{thm}[see \cite{Coifman1974}]
 Let $0<p\leq 1$. Then $f\in H^p(\R)$ if and only if there exist real 
numbers $a_i$,  atoms $b_i$ for $i\in \mathbb{N}$,  and  absolute constants $c, C$ such that such that
\[\ds f(\xi)=\sum_{i=0}^{\infty}a_ib_i(\xi)\quad \mbox{and}\quad  c\norm{f}_{H^p}\leq \sum_{i=0}^{\infty}\abs{a_i}^p \leq C \norm{f}_{H^p}\;.\]
\end{thm}
\noindent  Fefferman \cite{Fefferman1971} observed  that this result is in fact due to the duality between $H^1(\R)$ and the space of functions of bounded means oscillations (BMO), therefore providing another proof that the dual space of $H^1(\R)$ is in fact BMO. The era of the atomic decomposition therefore  started. One criticism of the atomic decomposition at the time was that it was too general, making it difficult or not very useful for applications. This atomic decomposition was proved to be quite useful in harmonic analysis. However, in an attempt to answer this criticism, Richard O'Neil and Geraldo De Souza  proposed an example of atoms defined on the interval $I=[0,1]$  that was  latter dubbed ``special atoms". This special atom has some very desirable properties as we will see in the sequel.

\begin{defn}\label{defn2}Consider $1\leq p<\infty$.
\begin{itemize}
\item[(a)]A  {\bf special atom of type 1} is a function $b:I\to \mathbb{R}$ such that
\[ b(\xi)=\left\{\begin{tabular}{ll}
$\dfrac{1}{|J|^{1/p}}\Intv{\chi_R(\xi)-\chi_L(\xi)}$, & \textrm{if $\xi\in J$} \\
1, &  \textrm{if $\xi\in I\backslash J$ }
\end{tabular}\right.\;,
\]
where $J$ is a subinterval of  $I$, $L$ and $R$ are the halves of $J$ such that $J=L\cup R$, and $|J|$ is the length of $J$\;.

\item[(b)]A  {\bf special atom of  type 2} is a function $c:J\to \mathbb{R}$ such that
\[c(\xi)=\dfrac{1}{|J|^{1/p}}\Intv{\chi_J(\xi)}\;,\]

where $J$ is an interval contained in $I$\;.
\end{itemize}
\end{defn}
\begin{remark} We observe that this definition can be extended on the unit ball of $\R^d$ by using dyadic decomposition, see \cite{Arcozzi2006}\;. 
\end{remark}
\noindent From this definition, they introduced the {\it special atom space} $B^p$ (for type 1 atom) defined on $J$, but with a different norm from the $L^p$-norm.
\begin{defn}
Let $1\leq p<\infty$. The special atom space $B^p$ (of type 1) is defined as 
\[ B^p=\set{f:I\to \mathbb{R}; f(\xi)=\sum_{n=0}^{\infty} \alpha_n b_n(\xi); \sum_{n=0}^{\infty} \abs{\alpha_n}<\infty}\;,\]
where the $b_n$'s are special atoms of type 1. The space $B^p$ is endowed with the norm
\[\norm{f}_{B^p}=\inf \sum_{n=0}^{\infty}\abs{\alpha_n},\] where the infimum is taken over all representations of $f$\;.
\end{defn}

\noindent The weighted special atom soon followed (see for instance \cite{DeSouza1989}) which gave rise a host of  very interesting properties, namely the analytic characterization. 
\begin{defn}\label{defweightedSpAtom}
 We define the {\bf weighted special atom (of type 1)} on $J$ as:
\[ b_w(\xi)=\frac{1}{w(J)}\left[ \chi_{_R}(\xi)-\chi_{_L}(\xi)\right]\;,\]
where \[w\in L^1(I)\quad \mbox{with}\quad w(J)=\int_Jw(\xi)d\xi, \quad \mbox{and $L,R$ are as in Definition \ref{defn2}}\;.\]
 The weighted special  atom space  is  the space $B_w$ of functions $f$ with atomic decomposition 
\[ f(\xi)=\sum_{n=0}^{\infty} \alpha_n b_{w,n}(\xi)\;,\]
endowed with the  Infimum norm.
\end{defn}
The importance of this definition can not be overstated. Indeed, weighted special atom spaces are invariant under  the Hilbert transform and they  contain some functions whose Fourier series diverge, see  \cite{DeSouza1994}. One of their most applicable features  is their connection to Haar wavelets, in that,  a Haar wavelet  function is just a special atom with weight $2^{-n/2}$,  \cite{Kwessi2019}. Weighted special atom spaces are also Banach equivalent to some Bergman-Besov-Lipschitz spaces (see \cite{DeSouza1985}), which leads  to a complete characterization of their lacunary functions, see  \cite{Kwessi2013}. Moreover, functions in $B_w$ have  analytic correspondences by integrating against   analytic functions whose real parts coincide with    the Poisson kernel. In particular, $B^1$ is Banach equivalent to the space of analytic functions $F$ on the complex unit disc for which $\ds F(z)=\frac{1}{2\pi}\int_0^{2\pi}\frac{e^{i\xi}+z}{e^{i\xi}-z}f(\xi)d\xi$. The question that was latter raised by Brett Wick in 2010 (personal communication with the first author) was whether this analytic  characterization could be achieved  in higher dimensions. In order to  entertain such a question, one has to,  for $d\geq2$,
\begin{enumerate}
\item first  provide a definition  of the  special atom and its weighted counterpart on $I^d$ so that its restriction to $I=[0,1]$ is the original special atom.
\item second, provide a definition of  the special atom space $B^p$ on  $I^d$.
\item third, verify that the Banach structure of $B^p$ is preserved.
\item fourth, set the conditions on the weight function $w$ on $I^d$ and define the weighted special atom  space $B_w$ on $I^d$.
\item fifth, define the analytic extension $F(\bm{z})$ of a function $f(\bm{\xi})\in B_w$ for\\ $\bm{z}=(z_1,z_2, \cdots, z_d)\in \mathbb{D}^d$ and $\bm{\xi}=(\xi_1, \xi_2,\cdots, \xi_d)\in I^d$.
\item sixth, verify that $B_w$ and its analytic extension $A_w^1$ are indeed Banach-equivalent under the above conditions.  
\end{enumerate}
\begin{remark}
 The first step  was recently accomplished in \cite{Kwessi2019}. Also the requirement that by restricting to $I=[0,1]$ we obtain the original special atom is for simplicity sake. The argument is important in high dimensions to prove for example in the case of Haar wavelets that we obtain an orthonormal system. However, there exist numerous ways to define atoms similar to the special atom. 
 \end{remark}
 We end this introductory part by recalling the definition of the weighted  Lipschitz class of functions.
 
 \begin{defn} Let $w$ be a weight function defined on $J=[a-h,a+h]\subseteq I$. The weighted Lipschitz class  is the class of continuous functions defined as 
 \[\Lambda_w=\set{f:\R \to \R: \norm{f}_{\Lambda_w}=\sup_{h>0,\xi}\;\abs{\frac{f(\xi+h)+f(\xi-h)-2f(\xi)}{w(J)}}<\infty.}\] 
 \end{defn}
 \noindent For completeness, recall that for $w(t)=t$, $\Lambda_w$ is the Zygmund class and for $w(t)=t^{\alpha}, 0<\alpha<2$, $\Lambda_w$ is the Lipschitz class of order $\alpha$. It  was proved in  \cite{DeSouza1989} that the dual space $B_w^*$ of $B_w$ is $\Lambda_w'=\set{f': f\in \Lambda_w}$, where $f'$ is understood in the sense of distributions.\\
 \noindent The remainder of the  paper   is organized as follows: In Section \ref{sect2}, we show how to extend weighted special atoms to high dimensions. In the last step,  we state  the  Main Theorem in Section \ref{mainresults}, and we will make concluding remarks in Section \ref{conclusion}.

\section{High Dimension Extension}\label{sect2}
Let $\bm{z}=(z_1,z_2,\ldots,z_d)\in \C^d$ for an integer $d\geq 1$. In fact, in the sequel, bold-faced symbols will represent vectors. We start out by proposing   a definition of  a weighted special atom in higher dimensions, for a general weight function $w$. When $w$ is the Lebesgue measure,  the interested reader can refer to \cite{Kwessi2019} for a more constructive approach in the definition.
\begin{defn}
Let  $\bm{\xi}=(\xi_1,\xi_2,\ldots, \xi_d) \in \R^d$ and $1\leq p<\infty.$
\begin{itemize}
\item[(a)]   Let $\ds J:=\prod_{j=1}^d[a_j-h_j,a_j+h_j]$ where $a_j, h_j$ are real numbers with $h_j>0$. The  weighted special atom (of type 1) on $J$, a sub-interval  of $I^d$, is defined as \[\ds b_w(\bm{\xi})=\frac{1}{w(J)}\left\{ \chi_{_{R}}(\bm{\xi})-\chi_{_{L}}(\bm{\xi})\right\},\] where $\ds w(J)=\int_Jw(\bm{\xi})d\bm{\xi}$  and $\ds R=\bigcup _{j=1}^{2^{d-1}}J_{i_j}$ for some $(i_1, i_2, \cdots, i_{2^{d-1}}\in \set{1,1, \cdots, 2^d})$ with $i_1<i_2<\cdots <i_{2^{d-1}}$ and $L=J\backslash R$. $\set{J_1, J_2,\cdots, J_{2^d}}$ is the collection of sub-cubes of $J$, cut by the hyperplanes $x_1=a_1, x_2, \cdots, x_d=a_d$, and  $\chi_A$ represents the characteristic function of set $A$. 
\item[(b)] The weighted special atom space  $B_w$ is the space of real-valued functions $f$ defined on $I^d$ such that 
\[f(\bm{\xi})=\sum_{n=0}^{\infty} \alpha_n b_{w,n}(\bm{\xi})\quad \mbox{with $\ds \sum_{n=0}^{\infty} \abs{\alpha_n}<\infty $},\]
endowed with the norm \[\ds \norm{f}_{B_w}=\inf \sum_{n=0}^{\infty} \abs{\alpha_n}\;,\]
where the infimum is taken over all possible representations of $f$.
\end{itemize}
\end{defn}

\noindent For example, for  $d=2$, for real numbers $a_1, a_2, h_1$, and $h_2$ such that  $h_1, h_2>0$,  consider a sub-interval $J$ of $I^d$ defined as  $J=[a_1-h_1, a_1+h_1]\times[a_2-h_2,a_2+h_2]\;.$
\noindent Let \begin{eqnarray*}
L_{1}&=&[a_1-h_1,a_1]\times[a_2-h_2,a_2], \quad L_{2}=[a_1-h_1,a_1)\times[a_2,a_2+h_2],\\
R_{1}&=& [a_1,a_1+h_1]\times[a_2-h_2,a_2),\quad R_{2}=(a_1,a_1+h_1]\times(a_2,a_2+h_2]\;.
\end{eqnarray*}
Consider  \[L=L_{1}\cup R_{2}\quad \mbox{and} \quad  R=L_{2}\cup R_{1}\;.\]
The special atom  $b(\xi_1,\xi_2)$ is then defined  as:
\begin{eqnarray*}\label{eq:bn(x,y)}
b_w(\xi_1,\xi_2)&=&\frac{1}{w(J)}\bigg\{ \chi_{_{R}}(\xi_1,\xi_2)-\chi_{_{L}}(\xi_1,\xi_2)\bigg\}\\
&=& \frac{1}{w(J)}\bigg\{ \chi_{_{L_{2}}}(\xi_1,\xi_2)+\chi_{_{R_{1}}}(\xi_1,\xi_2)-\chi_{_{L_{1}}}(\xi_1,\xi_2)-\chi_{_{R_{2}}}(\xi_1,\xi_2)\bigg\}\;.
\end{eqnarray*}
For $d\geq 2$, we consider  $\ds J=\prod_{j=1}^d[a_j-h_j,a_j+h_j]$ where $a_j, h_j$ are real numbers with $h_j>0$. \\
\noindent For $j=1,\cdots,d$,  we define $b_w(\bm{\xi})$ similarly. Figure \ref{fig1} below is an illustration of $b_w$ for $d=2$ (a) and $d=3$ (b) when $w$ is the Lebesgue measure.
\begin{figure}[H] 
   \centering
   \begin{tabular}{cc}
   {(a)} & { (b)}\\
   \includegraphics[scale=0.5]{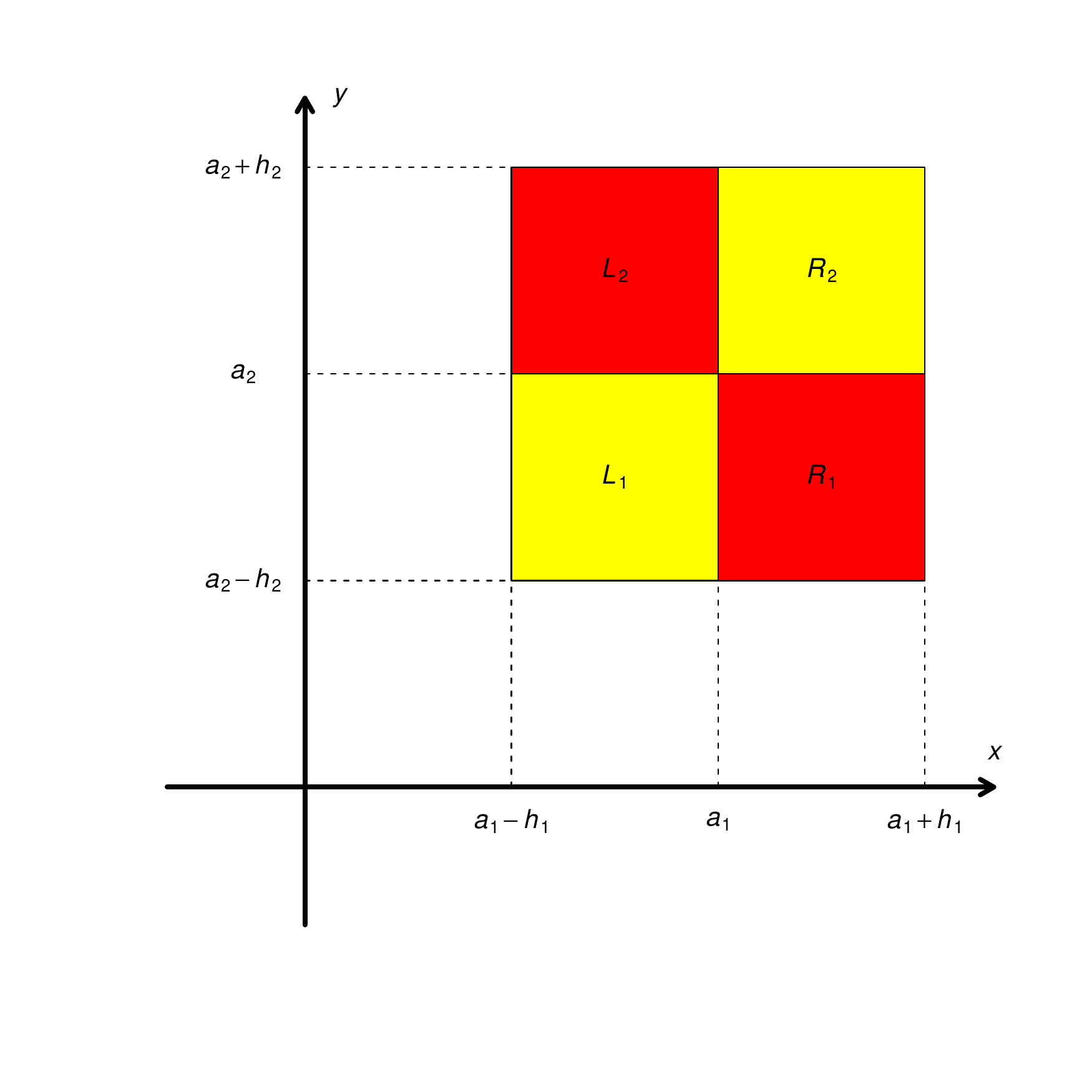} &
   \includegraphics[scale=0.5]{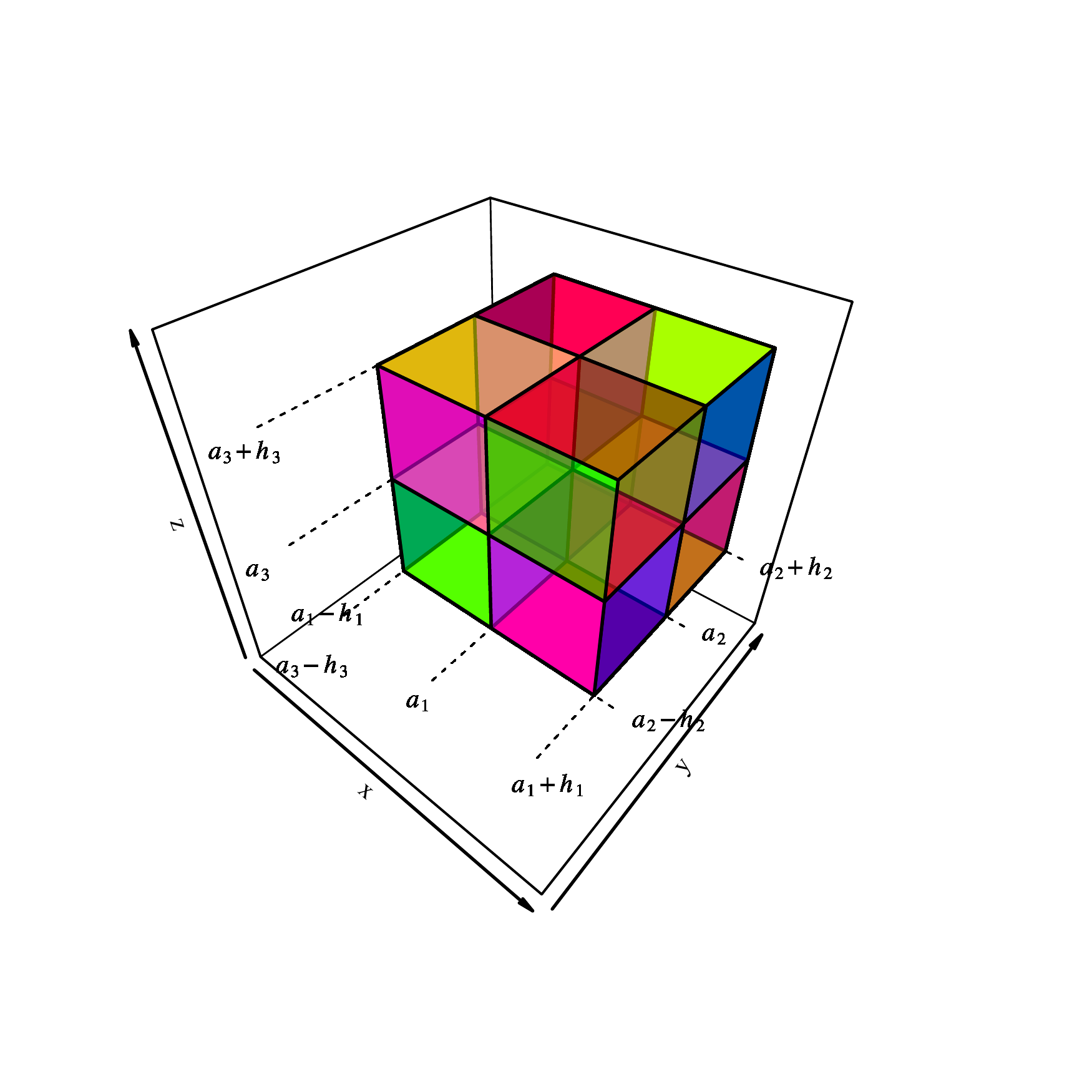} 
   \end{tabular}
   \caption{(a) represents a special Atom for $d=2$ and (b) represents the special atom for $d=3$. Note that for $d=3$, the color areas represent the different partitions of $J_j$ into subintervals $R$ and $L$ intervals.}
   \label{fig1}
\end{figure}

 With this definition, we can prove the following theorem about the Banach structure of $B_w$.
\begin{thm}\label{Theorem2.2} For $1\leq p<\infty$,
$\left(B_w, \norm{\cdot}_{B_w}\right)$ is a Banach space.
\end{thm}
\begin{proof} The proof can be seen in Section \ref{prooftheorem2.2} below.
\end{proof}
\begin{defn} \mbox{}\\
   Consider the function $P: \mathbb{D}^d\times I^d$ defined as \[\ds P(\bm{z},\bm{\xi})=\prod_{j=1}^d P_j(z_j,\xi_j)\quad \mbox{where} \quad P_j(z_j,\xi_j)=\frac{e^{i\xi_j}+z_j}{e^{i\xi_j}-z_j}\;.\] 
   \end{defn}
 \noindent We observe that for fixed $1\leq j\leq d$ and $z_j\in \mathbb{D}$,   $\mbox{Re}(P_j(z_j,\xi_j))$ is the Poisson Kernel. \\
  \noindent For a function $F$ defined on $\mathbb{D}^d$, we  define $F'(\bm{z})$ as 
 \[F'(\bm{z})=(f_1(\bm{z}), \cdots, f_d(\bm{z})), \quad \mbox{where}\quad f_j(\bm{z})=\frac{\partial F(\bm{z})}{\partial z_j}\;.\]
In the sequel $d\bm{\xi}=d\xi_1 d\xi_2\cdots d\xi_d$, and this  will be the case for similar  bold-faced symbols. Also for a set $A\subseteq S $, we will denote by $A^c=S\backslash A$.


\noindent Now we give the definition of special weights functions that will be necessary for the Proof of the main theorem
\begin{defn}
Let $w$ be a real-valued function defined  on $[0,1]$. Let $m$ and $n$ be positive integers. 
\begin{itemize}
\item[(a)] 
\noindent Then $w$ is said to be {\bf Dini} of order $m\geq 1$ and we denote $w\in \pD_{m}$ if   $\ds \frac{w(u)}{u^m}\in L^1(0,1)$, and there exists an absolute constant $C$ for which \[\ds \int_0^{u} \frac{w(\xi)}{\xi^m}d\xi\leq Cw\left(u\right),\quad \mbox{for $0<u<1$}\;.\]
\item[(b)] A function $w:[0,\infty) \to \R$ is said to be in the class $\pB_n$ for some positive integer $n$ and we denote $w\in \pB_n$ if for $0<u<1$,
\begin{enumerate}
\item  $w$ is increasing and $w(0)=0$\;.
\item There exists a constant $C$ such that $\ds \int_u^1 \frac{w(\xi)}{\xi^{n+1}}d\xi\leq C \frac{w(u)}{u^n}$\;.
\end{enumerate}
\item [(c)] Let $\bm{\xi}=(\xi_1,\xi_2,\cdots,\xi_d)\in \R^d$. For $1\leq j\leq d$, consider  weight functions $w_j$ defined on $\R_{+}$. We define the product  weight function $w(\bm{\xi})$ as 
\begin{equation}\label{eqn:weight} w(\bm{\xi})=\prod_{j=1}^d w_j(\xi_j),
\end{equation}
\item[(d)] A weight $w$ is said to be in the class $\mathcal{B}_p$ on $[0,1]$ if there exists a constant $C$ such that for any interval $J\subseteq [0,1]$ with center $\xi_J$, we have 
\[\frac{\abs{J}^p}{w(J)}\int_{J^c} \frac{w(\xi)}{\abs{\xi-\xi_J}^p}d\xi\leq C\;.\]
\end{itemize}
\end{defn}
\begin{remark}\mbox{}
\begin{enumerate}
 \item We observe that $w\in \mathcal{B}_p$ for some $p>1$ if and only  if  $w$ is a doubling measure, that is, there exists an absolute constant $C$ such that 
\[w[Q_{2h}(\xi)]\leq Cw[Q_h(\xi)],\quad \mbox{where $Q_h(\xi)=\{t\in \R: \abs{\xi-t}\leq h\}$}\;.\]
The class of doubling weights will be referred to as $\mathcal{D}$.
\item The  Muckenhoupt class of  weights, (see \cite{Muckenhoupt1972}) \[\mathcal{A}_p=\set{w \in L^1: \left(\frac{1}{\abs{I}}\int_I w(\xi)d\xi\right)\rb{\frac{1}{\abs{I}}\int_I w^{1/1-p}(\xi)d\xi}^{p-1}<\infty}\] is strictly contained in the class   $\mathcal{B}_p$.
\end{enumerate}
\end{remark}
\noindent The following is an important lemma relating class  $\pD_m, \pB_n$ and $\mathcal{D}$.
\begin{lemma}\label{Lem1} Let $w$ be a weight  function. Then for all  integer $ n\geq 1$, we have:
\[\pD_{1}\cap \pB_n\subseteq \mathcal{D}\;.\]

\end{lemma}
\begin{proof}
Fix $n\geq 1$ and let $w\in \pD_{1}\cap \pB_n$. We would like to show that for $u$ such that $2u\leq 1$ there exists an absolute constant $C$ such that \[\int_0^{2u} \frac{w(\xi)}{\xi}d\xi\leq C w\left(u\right)\;.\]
We have that \[ \int_0^{2u} \frac{w(\xi)}{\xi}d\xi=\int_0^{u} \frac{w(\xi)}{\xi}d\xi+\int_{u}^{2u} \frac{w(\xi)}{\xi}d\xi=I_1+I_2\;.\]
On one hand,  since $w\in \pD_{1}$,   there exists $C>0$ such that $I_1\leq C\cdot w\left(u\right)$.
On the other hand,
\begin{eqnarray*}
I_2= \int_{u}^{2u} \frac{w(\xi)}{\xi}d\xi &=& u^n \int_{u}^{2u} \frac{w(\xi)}{\xi u^n}d\xi\\
&\leq & u^n 2^{n+1} \int_{u}^{2u} \frac{w(\xi)}{\xi^{n+1}}d\xi, \quad \mbox{since $\frac{1}{\xi u^n}\leq \frac{1}{u^{n+1}} \leq \frac{2^{n+1}}{\xi^{n+1}}$}\\
&\leq &   C \cdot 2^{n+1}\cdot w\left(u\right), \quad \mbox{since $w\in \pB_n$}\\
&\leq & C\cdot w(u)\;.
\end{eqnarray*}
It follows that $I_1+I_2\leq C \cdot w(t)$ and thus $w\in \mathcal{D}$\;.
\end{proof}

\begin{defn} A product  weight  $w\in \pD_m$ (respectively $w\in \pB_n$ or $w\in \mathcal{B}_p$) if for each $1\leq j\leq d$,  the weight $w_j\in \pD_m$ (respectively $w_j\in \pB_n$ or $w_j\in \mathcal{B}_p$) for integers $m,n\geq 1$ and a real number $p>1$.
\end{defn}
We now introduce spaces of analytic functions that will be proven to be the analytic characterizations of  weighted special atom spaces.
\begin{defn}
Let $1\leq p<\infty$ be a real number. Given a function $w: \mathbb{D}^d\to \R_+$, we will  consider $A_w^p$  as  the space   of analytic functions $F$ defined on $\mathbb{D}^d$ such that 
\begin{equation}\label{eqn:AnalWeightedSpecialAtomSpace1}
\norm{F}_{A_w^p}=\abs{F(\bm{0})}+\frac{1}{(2\pi)^d}\int_{\mathbb{D}^d} \abs{F'(\bm{r}e^{i\bm{\xi}})}^pw(\bm{r}e^{i\bm{\xi}})d\bm{\xi}d\bm{r}<\infty,
\end{equation}
where 
\begin{equation}\label{eqn:Fderivative}
\abs{F'(\bm{r}e^{i\bm{\xi}})}=\norm{F'(\bm{r}e^{i\bm{\xi}})}_2=\sqrt{\abs{f_1}^2+\cdots+\abs{f_d}^2}\quad \mbox{and $\bm{z}=\bm{r}e^{i\bm{\xi}}$}\;.
\end{equation}
\end{defn}
\begin{remark}\mbox{}
\begin{enumerate}
\item We note that  for $d=1$ and $w(re^{i\xi})=1$, $A_w^1$  is contained in  the Hardy space $H^1(\mathbb{D}^d)$. 
\item By Holder's inequality, $A_w^p\subseteq A_w^1$ for $p>1$. In particular, $A_w^1$ contains   the weighted Dirichlet space $\mbox{D}_w$ (see \cite{IdrissiElFallah}) of analytic  functions $F$ such that 
\[\int_{\mathbb{D}} \abs{F'(re^{i\xi})}^2w(re^{i\xi})drd\xi<\infty\;.\]
\item \noindent When $r=\abs{z}^2$ and  $w(re^{i\xi})=(1-r)^{\alpha}=(1-\abs{z}^2)^{\alpha}$ for $ \alpha>0$, then $A_w^1$ contains the Bloch space $\mbox{B}_{\alpha}$ (see \cite{KeheZu}) of analytic functions $F$ such that 
\[\sup_{\mathbb{D}}\left\{ (1-\abs{z}^2)^{\alpha}\abs{F'(z)}\right\}<\infty\;.\]
\end{enumerate}
\end{remark}
\begin{defn}
Let $f\in B_w$. We define the   analytic extension $F$ of $f$ as  
 \begin{equation}\label{eqn1:analF}
\ds F(\bm{z})=\frac{1}{(2\pi)^d}\int_I P(\bm{z},\bm{\xi}) f(\bm{\xi}) d\bm{\xi}\;, 
\end{equation} 
in the sense that we can recover $f$ by taking the radial limit
\[f(\bm{\xi})=\lim_{\bm{r}\to 1}\mbox{Re}\;F(\bm{r}e^{i\bm{\xi}})\;,\]
 where the limit is taken  element-wise. The space $A_w^1$  defined in equation \eqref{eqn:AnalWeightedSpecialAtomSpace1}  for some $f\in B_w$ will be referred to  as the analytic extension of $B_w$.

\end{defn}
\begin{remark}\mbox{}
\begin{enumerate}
\item We observe that in the definition of the analytic extension $F$ in \eqref{eqn1:analF}, the function $\log(F(z))$ can be viewed as the outer function whereas $f(\xi)$ can be viewed as the inner function, similarly to  a Beurling factorization, see for example Definition 17.14 in \cite{Rudin1987}.\\
\item Additionally, in this definition, if one choose $f(\bm{\xi})=\log(1-\abs{b(\bm{\xi})}^2)^{\frac{1}{2}}$, then  we can define the function $a(z)$, the  so called-{\it Pythagorean mate of $b  \in \pH(b)$} (The interested reader can refer for example  to \cite{FricainMashreghi} for  more on these spaces) as 
\[a(\bm{z})=exp\rb{\int_I \frac{\bm{\xi}+\bm{z}}{\bm{\xi}-\bm{z}}\log(1-\abs{b(\bm{\xi})}^2)^{\frac{1}{2}}d\bm{\xi}}\;\]
This suggests that the space $A_w^p$ is closely related to the theory of $\pH(b)$ spaces, but more importantly, can be used as a gateway to their  study in higher dimensions. We know that there is an extensive literature on these spaces $\pH(b)$ in one dimension, see \cite{Beurling1948,FricainMashreghi, BlandigneresEtAl}.
\item \noindent Moreover, if $f\in L^p(\T^d)$ for $1<p<\infty$, then by Theorem 17.26 in  \cite{Rudin1987}, the analytic extension $F\in H(\mathbb{D}^d)$, the Hardy's space on the polydisk.
\end{enumerate}
\end{remark}
We will prove in the sequel  that when the weight function $w$ satisfies certain conditions, then the spaces $B_w$ and $A_w^1$ are in fact isometric to each other, that is, the inclusion operator $G: B_w\to A_w^1, G(f)=F$ is a Banach isometry.

\section{Main results}\label{mainresults}
Now we can now state our main theorem:
\begin{thm*}
Let $J_j=[a_j-h_j,a_j+h_j], \ds J=\prod_{j=1}^d J_j$. Let $w$ be a weight defined on J. Let  $A_{w}$ be the space of analytic functions defined above. Then we have the following: 

\begin{itemize}
\item[(a)]  $B_w\subseteq A_w^1$ if and only if  $w(\bm{r}e^{i\bm{\xi}})\equiv w(\bm{\xi})$ is a product weight and   $w\in \mathcal{B}_2$.
\item[(b)] $B_w$ is Banach equivalent to  $A_w^1$ if and only if $\ds w(\bm{r}e^{i\bm{\xi}})\equiv \frac{w(1-\bm{r})}{1-\bm{r}}$ is a product weight and $w\in \pD_1\cap \pB_2$.
\end{itemize}
\end{thm*}
\begin{remark}\mbox{}\\
The Main theorem essentially states that $B_w$ and $A_w^1$ are isomorphic as Banach spaces in  the sense that 
\begin{enumerate}
\item $B_w$ and $A_w^1$ are both Banach spaces,
\item $f\in B_w$ if and only if  its analytic extension $F\in A_w^1$,
\item $F\in A_w^1$ if and only if  $\ds\lim_{\bm{r}\to 1}\mbox{Re}\; F(\bm{r}e^{i\bm{\xi}})\in B_w$,
\item $\norm{f}_{B_w}\equiv \norm{F}_{A_w^1}$\;.
\end{enumerate}
\end{remark}
\begin{remark}\mbox{}\\
\begin{enumerate}
\item In the first part of the Main Theorem, the weight function depends only on the argument $\bm{\xi}$ of $\bm{z}=\bm{r}e^{i\bm{\xi}}$, whereas in the second part, it depends only on the radius $\bm{r}$. The condition that $w\in \mathcal{B}_2$ in the first part is weaker than the condition $w\in \pD_1\cap \pB_2$ in the second part since per Lemma \ref{Lem1}, 
$  w\in \pD_1\cap \pB_2\subseteq \mathcal{D}=\bigcup\limits_{p>1}\mathcal{B}_p $
implies the existence of $p>1$ such that $w \in \mathcal{B}_p$. There is no guarantee that this $p$ will be 2 as in the first part. However, what the two parts have in common is the necessary condition that  if $B_w$ is contained in $A_w^1$, then  the weight $w\in \mathcal{D}$.
\item In the Main Theorem, the weight $w$ is a product weight, however general weights $w$ defined on $I^d$ are not addressed in this manuscript and it would be a worthwhile future endeavor to have a holistic understanding of the role of the weight $w$.
\end{enumerate}
\end{remark}
The proof of the Main Theorem relies on some  crucial lemmas  that will be stated below. The first lemma shows that partial derivatives of analytic extensions of special atoms in higher dimensions are bounded. Henceforth, the  constants $C$ will be generic and when necessary, their dependence on an interval $J$ will be specified accordingly.
\begin{lemma}\label{lemma1}
Let $J_j=[a_j-h_j,a_j+h_j], \ds J=\prod_{j=1}^d J_j$ and $\ds F(z)=\frac{1}{(2\pi)^d}\int_J P(\bm{z},\bm{\xi})b_w(\bm{\xi})d\bm{\xi}$. Then for any $j=1,\cdots, d$\;,  
\begin{itemize}
\item[(1)] there exists a constant $C(J)$ such that 
\begin{equation}\label{derivative}
f_j(z_j)=C(J)K_1(a_j,h_j,z_j)\prod_{\underset{l\neq j}{l=1}}^d K_2(a_l,h_l,z_l)\;,\end{equation}
where 
\begin{eqnarray*}
K_1(a_j,h_j,z_j)&=& \frac{1}{i}\left[\frac{1}{z_j-e^{i(a_j-h_j)}}+\frac{1}{z_j-e^{i(a_j+h_j)}} +\frac{2}{e^{ia_j}-z_j}\right]\;,\\
K_2(a_l,h_l,z_l)&=& \frac{2}{i}\left[ \ln\left(e^{i(a_l-h_l)}-z_l\right)+\ln\left(e^{i(a_l+h_l)}-z_l\right)-2\ln\left(e^{ia_l}-z_l\right)\right]\;.
\end{eqnarray*}
\item[(2)] Moreover for $i,j=1,\cdots, k$, there are absolute constants $C_1$ and $C_2$ such that 
\[\abs{K_1(a_j,h_j,z_j)}\leq C_1,\quad \abs{K_2(a_l,h_l,z_l)}\leq C_2\;.\]
\end{itemize}
\end{lemma}

\begin{lemma} \label{lemma3} Let real numbers $a$ and $h>0$, and $z\in \mathbb{D}$. Let $J=[a-h,a+h]$.
\begin{itemize}
\item[(a)] If $w \in \mathcal{B}_2$ such that $w(re^{i\xi})\equiv w(\xi)$, then there exists a constant $C$ such that
\[\int\int_{\mathbb{D}}\abs{K_1(a,h,z)}w(\xi)d\xi dr\leq C(J)<\infty\;.\]
\item[(b)] If $w \in \pD_1\cap \pB_2$ such that $\ds w(re^{i\xi})\equiv \frac{w(1-r)}{1-r}$, then there exists a constant $C$ such that
\[\int\int_{\mathbb{D}}\abs{K_1(a,h,z)}\frac{w(1-r)}{1-r}d\xi dr\leq C(J)<\infty\;.\]
\end{itemize}
\end{lemma}

\begin{lemma} \label{lemma4} Let $1\leq j\leq d$ and  $J_j=[a_j-h_j,a_j+h_j]$ for  real numbers $a_j$ and $h_j>0$. Consider  $z_j=r_je^{i\xi_j}\in \mathbb{D}$ such that $h_j<\abs{e^{ia_j}-z_j}$.  Consider a product weight $w$ such that  $\frac{w_j(t)}{t^2}\in L^1(0,1)$ for all $1\leq j\leq d$. Then  there exists a constant  $C(J_j)>0$ such that 
\[ \frac{h_j^2}{w_j(J_j)}\int_{\xi_j \notin J} \frac{\omega(\xi_j)}{\xi_j^2}d\xi_j\leq C(J_j) \int\int_{\mathbb{D}}\abs{f_j(\bm{z})}d\xi_jdr_j.\]
 \end{lemma}
 \begin{lemma}\label{lemma5}
 Let   $j=[a-h,a+h]$ for  real numbers $a$ and $h>0$ and $w(t)/t$ and in $L^1(J)$. Fix $1\leq j\leq d$.
 \begin{itemize}
 \item[(a)] Consider  
 $D_{1}=\set{z=re^{i\xi}\in \mathbb{D}: h<\abs{e^{ia}-z}}$. There exists an absolute constant $C$ such that 
 \[\int\int_{D_{1}}\abs{f_j(\bm{z})} \frac{w(1-r)}{1-r}d\xi dr \geq C_j \int_{h}^1 \frac{w(u)}{u}du\]
 
 \item[(b)] Consider the subset $D_{2}=\set{z\in  \mathbb{D}: \abs{e^{ia}-z}\leq \frac{h}{4}}.$ There exists an absolute constant $C_j$ such that \[ \int\int_{D_{2}}\abs{f_j(\bm{z})} \frac{w_j(1-r)}{1-r}d\xi dr\geq C_j \int_0^{\frac{h}{4\sqrt{2}}} \frac{w(u)}{u} du\;.\]
 \end{itemize}
 \end{lemma}
 

\begin{proof}[{\bf Proof of the Main Theorem}]\mbox{}\\
Part (a):  Let $w$ be a product weight such that $w\in \mathcal{B}_2$. To show that $B_w\subseteq A_w^1$, it will be enough to show that  analytic extensions of   weighted special atoms $b_w(\bm{\xi})$ are contained in $A_w^1$, that is, we will show that for a special atom $b_w(\bm{\xi})\in B_w$, we have $F\in A_w^1$  where \[ F(\bm{z})=\frac{1}{(2\pi)^d}\int_I P(\bm{z},\bm{\xi})b_w(\bm{\xi})d\bm{\xi}.\]

\noindent Fix $1\leq j\leq d$. Let $d\bm{\xi}^{-j}d\bm{r}^{-j}=d\xi_1dr_1\cdots d\xi_{j-1}dr_{j-1}d\xi_{j+1}dr_{j+1}\cdots d\xi_ddr_d$.
\noindent From Lemma \ref{lemma1} above, we have 
\begin{eqnarray*}
\abs{f_j(\bm{z})}&=&C(J)\abs{K_1(a_j,h_j,z_j)}\prod_{\underset{l\neq j}{l=1}}^d \abs{K_2(a_l,h_l,z_l)}\\
&\leq & M(J)\abs{K_1(a_j,h_j,z_j)}\quad\mbox{where}
\quad M(J)=C(J)C_2^{d-1}\;.
\end{eqnarray*}
Also from  Lemma \ref{lemma1} above, $K_1(a_j,h_j,z_j)$ is bounded, for all $1\leq j\leq d$.
Therefore  $\ds \sup_{1\leq j\leq d}\abs{K_1(a_j,h_j,z_j)} $ exists. Moreover using  the definition of $\abs{F'(\bm{z})}$ in  equation \eqref{eqn:Fderivative}, we obtain
\begin{eqnarray}\label{ineq:Fprime}
\abs{F'(\bm{z})}&\leq &d^{1/2}M(J)\sup_{1\leq j\leq d}\abs{K_1(a_j,h_j,z_j)}\;. 
\end{eqnarray}
We then have that 
\begin{equation*}
\begin{aligned}
\int\int_{\mathbb{D}^d}\abs{F'(\bm{z})}w(\bm{\xi})d\bm{\xi}d\bm{r}={} &\int\int_{\mathbb{D}^d}\abs{F'(\bm{z})}\left(\prod_{j=1}^d w_j(\xi_j)\right)d\bm{\xi}d\bm{r}\\
& \leq  d^{1/2}M(J)\left(\sup_{1\leq j\leq d} \int\int_{\mathbb{D}} \abs{K_1(a_j,h_j,z_j)}w_j(\xi_j)d\xi_j dr_j\right)\\
& \times\left( \int\int_{\mathbb{D}^{d-1}} \left(\prod_{\underset{l\neq j}{l=1}}^d w_l(\xi_l)\right)d\bm{\xi}^{-j}d\bm{r}^{-j}\right)\\
& \leq C\sup_{1\leq j\leq d}\left( \int\int_{\mathbb{D}} \abs{K_1(a_j,h_j,z_j)}w_j(\xi_j)d\xi_j dr_j\right)<\infty\;.
\end{aligned}
\end{equation*}
This proves that $F\in A_w^1$.  \\
\noindent Conversely, suppose that $F\in A_w^1$. We will show that in this case, $w\in \mathcal{B}_2$. As above, it suffices to consider analytic extensions $F$ of weighted special atoms.  Let $C>0$ such that $\norm{F}_{A_w^1}<C$. 

\noindent Fix $1\leq j\leq d$. Then we know that $\norm{F'(\bm{z})}_2\geq \norm{F'(\bm{z})}_{\infty}\geq \abs{f_j(\bm{z})}$. Therefore
\begin{equation*}
\begin{aligned}
C>\int\int_{\mathbb{D}^d}\abs{F'(\bm{z})} w(\bm{\xi})d\bm{\xi}d\bm{r}&=  \int\int_{\mathbb{D}^d}\left(\sum_{l=1}^d \abs{f_l(\bm{z})}^2\right)^{1/2}w(\bm{\xi})d\bm{\xi}d\bm{r}\\
&\geq \left( \int\int_{\mathbb{D}^{d-1}} \left(\prod_{\underset{l\neq j}{l=1}}^d w_l(\xi_l)\right)d\bm{\xi}^{-j}d\bm{r}^{-j}\right)\\
&\times \left( \int\int_{\mathbb{D}}\abs{f_j(\bm{z})} w_j(\xi_j)d\xi_jdr_j\right)\\
&\geq\left(\prod_{\underset{l\neq j}{l=1}}^d \norm{w_l}_{L^1}\right)\times\left( \int_{\xi_j\notin J_j} \frac{w_j(\xi_j)}{\xi_j^2}d\xi_j\right), \quad \mbox{by  Lemma \ref{lemma4}}
\end{aligned}
\end{equation*}
Since $j$ is arbitrary, it follows that 
\[\frac{\abs{J_j}^p}{w_j(J_j)}\int\int_{\xi_j\notin J_j} \frac{w_j(\xi_j)}{\xi_j^2}d\xi_jdr_j<C, \quad \forall j=1, \cdots, d.\]
Therefore $w\in \mathcal{B}_2$ and this concludes the proof of part (a) of the theorem.\\
\noindent Part  (b). Now assume $w$ is a product weight such that  $w\in \pD_1\cap \pB_2$. Let $F \in B_w$. As above, we will proceed by showing that analytic extensions $F$ of special atoms are in $A_w^1$. Thus consider $\ds F(z)=\frac{1}{(2\pi)^d}\int_J P(\bm{z},\bm{\xi})b_w(\bm{\xi})d\bm{\xi}$. We know from above  equation \eqref{ineq:Fprime} that
\begin{equation*}
\begin{aligned}
\int\int_{\mathbb{D}^d}\abs{F'(\bm{z})}\frac{w(1-\bm{r})}{1-\bm{r}}d\bm{\xi}d\bm{r}
& \leq  d^{1/2}M(J)\left(\sup_{1\leq j\leq d} \int\int_{\mathbb{D}} \abs{K_1(a_j,h_j,z_j)}\frac{w_j(1-r_j)}{1-r_j}d\xi_j dr_j\right)\\
& \times\left( \int\int_{\mathbb{D}^{d-1}} \left(\prod_{\underset{l\neq j}{l=1}}^d \frac{w_l(1-r_l)}{1-r_l}\right)d\bm{\xi}^{-j}d\bm{r}^{-j}\right)
\end{aligned}
\end{equation*}
By Lemma \ref{lemma3}, we have
\[\sup_{1\leq j\leq d} \int\int_{\mathbb{D}} \abs{K_1(a_j,h_j,z_j)}\frac{w_j(1-r_j)}{1-r_j}d\xi_j dr_j\leq \sup_{1\leq j\leq d} C(J_j)<\infty\;.\]
By the Dini condition, we have that
\[
C_0= \int\int_{\mathbb{D}^{d-1}} \left(\prod_{\underset{l\neq j}{l=1}}^d \frac{w_l(1-r_l)}{1-r_l}\right)d\bm{\xi}^{-j}d\bm{r}^{-j}= (2\pi)^{d-1} \prod_{\underset{l\neq j}{l=1}}^d \left(\int_0^{1} \frac{w_l(u_l)}{u_l}du_l\right)<\infty\;.\]
Therefore, we can infer that $F\in A_w^1$.\\
\noindent Conversely, suppose that $F\in A_w^1$. We will show that $w\in \pD_1\cap \pB_2$. Let $C>0$ such that $\norm{F}_{A_w^1}\leq C$.  Fix $1\leq j\leq d$. Then as above,
\begin{equation}\label{eqnlowedbd}
\begin{aligned}
C>\int\int_{\mathbb{D}^d}\abs{F'(\bm{z})} \frac{w(1-\bm{r})}{1-\bm{r}}d\bm{\xi}d\bm{r}
&\geq \left( \int\int_{\mathbb{D}^{d-1}} \left(\prod_{\underset{l\neq j}{l=1}}^d \frac{w_l(1-r_l)}{1-r_l}\right)d\bm{\xi}^{-j}d\bm{r}^{-j}\right)\\
&\times \left( \int\int_{\mathbb{D}}\abs{f_j(\bm{z})} \frac{w_j(1-r_j)}{1-r_j}d\xi_jdr_j\right)\\
&= C_0\times \int\int_{\mathbb{D}}\abs{f_j(\bm{z})} \frac{w_j(1-r_j)}{1-r_j}d\xi_jdr_j\;.\\
\end{aligned}
\end{equation}
We can first combine the latter with equation  \eqref{eqnlowedbd} and  Lemma \ref{lemma5} (a) to  obtain that $\ds \int_{h_j}^{1} \frac{w_j(u_j)}{u_j^3}dr_j\leq C_j$, that is, $w_j\in \pB_2$. Since $j$ is arbitrary, it follows that $w\in \pB_2$\;.
 We can also combine the latter  with  \eqref{eqnlowedbd},  Lemma \ref{Lem1},  and Lemma \ref{lemma5} (b) to conclude that $w_j\in \pD_1$. Since $j$ is arbitrary, it follows that $w\in \pD_1$\;.\\
It remains to show that $B_w$ and $A_w^1$ are norm-equivalent. Since $B_w\subseteq A_w^1$, there exists a constant  $M>0$ such that $\norm{f}_{A_w^1}\leq M\norm{f}_{B_w}$. To obtain the reverse inequality, it suffices to use a simple extension of the one dimension case to obtain that  the dual $B_w^*$ of $B_w$ is  continuously contained in the dual  $A_w^{1*}$ of $A_w^1$. Hence by virtue of  the inclusion $B_w\subseteq A_w^1$, we have $A_w^{1*}\subseteq B_w^*$, so that $A_w^{1*}= B_w^*$. We then  have the following situation:\\
\noindent (a): $B_w\subseteq A_w^1$ implies that the inclusion map $G: B_w\to A_w^1$ is an open map.\\
\noindent (b): $\norm{f}_{A_w^1}\leq M\norm{f}_{B_w}$ implies that $G$ is a bounded linear map.

Thus by the Open Mapping Theorem, the range of $G(B_w)=B_w$ is dense in $A_w^1$. \\
\noindent (c) Since $A_w^{1*}= B_w^*$, it follows that $B_w$ and $A_w^1$ are norm-equivalent, see for example \cite{Folland1999} page 160. 


\end{proof}
 \begin{proof}[{\bf Proof of Theorem  \ref{Theorem2.2}}] \label{prooftheorem2.2}
 In the proof that $\norm{\cdot}_{B_w}$ is a norm, only the triangle inequality requires special care. Using the definition of the infimum, let $\epsilon>0$ and let $\set{\alpha_n}_{n\in \N},  \set{\beta_n}_{n\in \N}$ such that 
 $\ds f(\xi)=\sum_{n\in \N} \alpha_n b_{w,n}(\xi)$ and  $\ds g(\xi)=\sum_{n\in \N} \beta_n b_{w,n}(\xi)$ and $\ds \sum_{n \in \N}\abs{\alpha_n}<\norm{f}_{B_w}+\epsilon/2, ~\sum_{n \in \N}\abs{\beta_n}<\norm{g}_{B_w}+\epsilon/2$.  
Hence  $ \ds (f+g)(\xi)=\sum_{n\in \N}(\alpha_n+\beta_n) b_{w,n}(\xi)$ with $\ds \sum_{n\in \N}\abs{\alpha_n+\beta_n}\leq \sum_{n\in \N}\abs{\alpha_n}+\sum_{n\in \N}\abs{\beta_n}<\infty\;.$ Therefore,
 \[\ds \norm{f+g}_{B_w}\leq \sum_{n\in \N}\abs{\alpha_n+\beta_n}\leq \sum_{n \in \N}\abs{\alpha_n}+\sum_{n \in \N}\abs{\beta_n}<\norm{f}_{B_w}+\norm{g}_{B_w}+\epsilon\;. \]
 Since $\epsilon$ is arbitrary, it follows that $\norm{f+g}_{B_w}\leq \norm{f}_{B_w}+\norm{g}_{B_w}$.\\
 Now, let us prove that $B_w$ is a Banach space. It will be sufficient to show that every absolutely convergent sequence is convergent. In short, it will be enough to show that given a sequence $\set{f_n}_{n\in \N}$, we have $\ds \norm{\sum_{n\in \N} f_n}_{B_w}\leq \sum_{n\in \N}\norm{f_n}_{B_w}$.\\
 Let $\epsilon>0$. Given $n\in \N$, there is a  sequence $\alpha_{n_k}$ of real numbers such that  $\ds f_n(\xi)=\sum_{k\in \N}\alpha_{n_k}b_{w,n_k}(\xi)$ with $\ds \sum_{k\in \N}\abs{\alpha_{n_k}}<\norm{f_n}_{B_w}+\frac{\epsilon}{2^n}$\;.
 Therefore
 \[\sum_{n\in \N}\sum_{k\in \N}\abs{\alpha_{n_k}}<\sum_{n\in \N}\norm{f_n}_{B_w}+\sum_{n\in \N}\frac{\epsilon}{2^n}=\sum_{n\in \N}\norm{f_n}_{B_w}+\epsilon\;.\]
 Since $\epsilon$ is arbitrary, it follows that 
 \[\norm{\sum_{n\in \N}f_n}\leq \sum_{n\in \N}\norm{f_n}_{B_w}\;.\]
 \end{proof}

\begin{proof}[{\bf Proof of Lemma \ref{lemma1}}]
Let $\ds P_j'(z_j,\xi_j)=\frac{\partial P_j(z_j,\xi_j)}{\partial z_j}=\frac{e^{i\xi_j}}{(e^{i\xi_j}-z_j)^2}$, for $j=1,2,\ldots,d$.\\
Let us start with $d=2$. Let $J=[a_1-h_1,a_1+h_1]\times[a_2-h_2,a_2+h_2]$.\\
Then 
\[  F(z_1,z_2)=\frac{1}{(2\pi)^2}\int_J P_1(z_1,\xi_1)P_2(z_2,\xi_2)b_w(\xi_1,\xi_2)d\xi_1d\xi_2,\]
and 
\begin{eqnarray*}
f_1(z_1,z_2)&=&\frac{2}{(2\pi)^2}\int_J P_1'(z_1,\xi_1)P_2(z_2,\xi_2)b_w(\xi_1,\xi_2)d\xi_1d\xi_2\\
&=&\frac{2}{w(J)(2\pi)^2}[I_1+I_2-I_2-I_4]\;,
\end{eqnarray*}
where 
\begin{eqnarray*}
I_1&=& \int_{a_1-h_1}^{a_1}\int_{a_2}^{a_2+h_2} P_1'(z_1,\xi_1)P_2(z_2,\xi_2)d\xi_1d\xi_2, \quad I_2= \int_{a_1}^{a_1+h_1}\int_{a_2-h_2}^{a_2} P_1'(z_1,\xi_1)P_2(z_2,\xi_2)d\xi_1d\xi_2\\
I_3&=& \int_{a_1-h_1}^{a_1}\int_{a_2-h_2}^{a_2} P_1'(z_1,\xi_1)P_2(z_2,\xi_2)d\xi_1d\xi_2, \quad I_4= \int_{a_1}^{a_1+h_1}\int_{a_2}^{a_2+h_2} P_1'(z_1,\xi_1)P_2(z_2,\xi_2)d\xi_1d\xi_2\;.
\end{eqnarray*}
Therefore
\begin{eqnarray*}
I_1-I_4&=& \int_{a_1-h_1}^{a_1}\int_{a_2}^{a_2+h_2} P_1'(z_1,\xi_1)P_2(z_2,\xi_2)d\xi_1d\xi_2- \int_{a_1}^{a_1+h_1}\int_{a_2}^{a_2+h_2} P_1'(z_1,\xi_1)P_2(z_2,\xi_2)d\xi_1d\xi_2\\
&=& -K_1(a_1,h_1,z_1)M_2(a_2,h_2,z_2)\;,
\end{eqnarray*}
where \begin{eqnarray*}
K_1(a_1,h_1,z_1)&=-&\int_{a_1-h_1}^{a_1} \frac{e^{i\xi_1}}{(e^{i\xi_1}-z_1)^2}d\xi_1+\int_{a_1}^{a_1+h_1}\frac{e^{i\xi_1}}{(e^{i\xi_1}-z_1)^2}d\xi_1\\
&=& \frac{1}{i}\left[ \frac{1}{z_1-e^{i(a_1-h_1)}}+\frac{1}{z_1-e^{i(a_1+h_1)}}+\frac{2}{e^{ia_1}-z_1}\right]\;.
\end{eqnarray*}
Also,
\begin{eqnarray*}
M_2(a_2,h_2,z_2)&=&\int_{a_2}^{a_2+h_2} P_2(z_2,\xi_2)d\xi_2\\
&=& \frac{1}{i}\left[ -ih_2+2\ln(e^{ia_2}-z_2)-2\ln(e^{i(a_2-h_2)}-z_2)\right]\;.
\end{eqnarray*}
Likewise, we have 
$I_1-I_3=K_1(a_1,h_1,z_1)M_2'(a_2,h_2,z_2)$ with \[M_2'(a_2,h_2,z_2)=\frac{1}{i}\left[ -ih_2-2\ln(e^{ia_2}-z_2)+2\ln(e^{i(a_2-h_2)}-z_2)\right]\;.\]
It follows that 
\[I_1-I_4+I_2-I_3=K_1(a_1,h_1,z_1)\left[M_2'(a_2,h_2,z_2)-M_2(a_2,h_2,z_2)\right]=K_1(a_1,h_1,z_1)K_2a_2,h_2,z_2),\] 
where 
\[K_2(a_2,h_2,z_2)=\frac{2}{i}\left[\ln(e^{i(a_2-h_2)}-z_2)+\ln(e^{i(a_2+h_2)}-z_2)-2\ln(e^{ia_2}-z_2)\right]\;.\]
Hence
\[\ds f_1(z_1,z_2)=C(J)K_1(a_1,h_1,z_1)K_2(a_2,h_2,z_2),\]
 where $\ds C(J)=\frac{2}{w(J)(2\pi)^2}$. 
Similarly, we obtain
\[\ds f_2(z_1,z_2)=C(J)K_1(a_2,h_2,z_2)K_2(a_1,h_1,z_1)\;.\]
For $d\geq 2$, we observe that the constant $C(J)$ will remain the same, regardless of the variable of  differentiation. Moreover, the function $K_1$ takes as arguments  $a_j, h_j,  z_j$ if we are differentiating with respect to $z_j$  and $K_2$ takes as arguments $a_l, h_l, h_l$, for all  $l\neq j$. The product comes from the fact that the integrand is made of functions with separable variables. Hence, we conclude that 
\[f_j(\bm{z})=C(J)K_1(a_j,h_j,z_j)\prod_{\underset{l\neq k}{l=1}}^d K_2(a_l,h_l,z_l)\;.\]
Moreover,  \[ \abs{K_2(a_l,h_l,z_l)}\leq 2\left[ \abs{\ln(e^{i(a_l-h_l)}-z_l)}+\abs{\ln(e^{i(a_l+h_l)}-z_l)}+2\abs{\ln(e^{i(a_l)}-z_l)}\right]\;.\]
Put $Z_{1l}=e^{i(a_l-h_l)}-z_l$. We know that $|z_l|<1$,  thus $\ln(\abs{Z_{il}})\leq \ln(2)$.
  \[\abs{\ln(Z_{1l})}=\sqrt{(\ln(\abs{Z_{il}}))^2+\mbox{arg}(Z_{1l})^2}.\]
  \noindent Consequently, 
  $\abs{\ln(Z_{1l})}\leq \sqrt{(\ln(2))^2+\frac{\pi^2}{4}}$. Applying a similar argument to $\abs{\ln(e^{i(a_l-h_l)}-z_l)}$ and $\abs{\ln(e^{i(a_l-h_l)}-z_l)}$, we obtain that 
  \[\abs{K_2(a_l,h_l,z_l)}\leq 3\sqrt{4(\ln(2))^2+\pi^2}\;.\]
  We will use a similar argument to \cite{DeSouza1989} 
  to deal with $K_1(a_1,h_1,z_1)$. However, this  approach is much general than theirs in that they assumed that  $a_1=0$ which is not assumed here.  We  observe  that 
  \begin{eqnarray*}
  iK_1(a_1,h_1,z_1)&=& \frac{1}{z_1-e^{i(a_1-h_1)}}+\frac{1}{z_1-e^{i(a_1+h_1)}}+\frac{1}{e^{ia_1}-z_1}\\
  &=& \frac{2e^{ia_1}(z_1+e^{ia_1})(1-\cos h_1)}{(z_1-e^{i(a_1-h_1)})(z_1-e^{i(a_1+h_1)})(e^{ia_1}-z_1)}\;.
  \end{eqnarray*}
 We have that 
 \begin{eqnarray*}
 (z_1-e^{i(a_1-h_1)})(z_1-e^{i(a_1+h_1)})&=& ((e^{ia_1}-z_1)^2 +2e^{ia_1}z_1(1-\cos h_1),
 \end{eqnarray*}
 so the modulus of the denominator of $ iK_1(a_1,h_1,z_1)$ is 
 \begin{eqnarray*}
 \abs{(z_1-e^{i(a_1-h_1)})(z_1-e^{i(a_1+h_1)})(e^{ia_1}-z_1)}&=&\abs{(e^{ia_1}-z_1)^2+2e^{ia_1}z_1(1-\cos h_1)}\abs{e^{ia_1}-z_1}\\
  &\geq & \abs{\left(e^{ia_1}-z_1\right)^2-h_1^2}\abs{e^{ia_1}-z_1}\;.
 \end{eqnarray*}
 The last inequality is obtained by noticing that $\abs{z_1}<1, \abs{e^{ia_1}}=1$, and $1-\cos h_1\leq \frac{h_1^2}{2}$. Now consider $D_1=\{z_1\in \mathbb{D}: \abs{e^{ia_1}-z_1}>2h_1\}$. For $z_1\in D_1$, the last inequality implies that 
 \[\abs{(z_1-e^{i(a_1-h_1)})(z_1-e^{i(a_1+h_1)})(e^{ia_1}-z_1)}\geq \frac{3}{4}\abs{e^{ia_1}-z_1}^3\geq 6h_1^3.\]
 On the other hand, the modulus of the numerator of $ iK_1(a_1,h_1,z_1)$ is bounded by 2$h_1^2$ on $D_1$, so that on $D_1$, one has 
 \begin{equation}\label{eqn:bbK1}
 \abs{ K_1(a_1,h_1,z_1)}\leq \frac{8}{3}\frac{h^2}{\abs{e^{ia_1}-z_1}^3}\leq \frac{1}{3h_1}\;.
 \end{equation}

 \noindent Now let $z_1\in D_1^c$. Then we have that \[\abs{z_1-e^{i(a_1-h_1)}},\abs{z_1-e^{i(a_1+h_1)}},\abs{e^{ia_1}-z_1}\leq 4h_1\;.\]
 Put \[\Phi_n^*=\set{z_1\in \mathbb{D}: 2^{1-n}h_1< \abs{e^{ia_1}-z_1}\leq 2^{2-n}h_1}=\set{z_1\in \mathbb{D}: \frac{2^{n-2}}{h_1} \leq \frac{1}{\abs{e^{ia_1}-z_1}} < \frac{2^{n-1}}{h_1}}\;.\]
 We note that $(0,1]=\bigcup\limits_{n=0}^{\infty}\left(\frac{1}{2^{n+1}},\frac{1}{2^n}\right]$. 
  It follows that 
 \begin{equation}\label{eqn:Dcomplement}
 \begin{aligned}
 D_1^c& \subseteq \set{z_1\in \mathbb{D}: \abs{e^{ia_1}-z_1}\leq 4h_1}\\
 &\subseteq \bigcup\limits_{n=0}^{\infty}\set{z_1\in \mathbb{D}: \frac{4h_1}{2^{n+1}}< \abs{e^{ia_1}-z_1}\leq \frac{4h_1}{2^{n}}}\\
 &= \bigcup\limits_{n=0}^{\infty}\Phi_n^*\;.
 \end{aligned}
 \end{equation}
Thus, there exists an integer $n$ such that 
\[\abs{K_1(a_1,h_1,z_1)}\leq \frac{1}{\abs{z_1-e^{i(a_1-h_1)}}}+\frac{1}{\abs{z_1-e^{i(a_1+h_1)}}}+\frac{1}{\abs{e^{ia_1}-z_1}}\leq 3 \frac{2^{n-1}}{h_1}\;.\]
We conclude by taking $C_1=\max\{3 \frac{2^{n-1}}{h_1}, \frac{1}{3h_1}\}$ and $C_2=3\sqrt{4(\ln(2))^2+\pi^2}$\;.
\end{proof}
\begin{proof}[{\bf Proof of Lemma \ref{lemma3}}]\mbox{}\\

\noindent Part (a): 
Put $J=[a-h,a+h]$ for some real numbers $a$ and $h>0$. Let $N$ be the smallest integer such that $2^Nh\geq 1$. Then for all $n\leq N$, and $z\in D_1$, we have $2^nh<\abs{e^{ia}-z}<2<2^{n+1}h$. \\
\noindent Also we observe that  $\set{ z=re^{i\xi}\in \mathbb{D}:  \abs{e^{ia}-z}\leq \nu}\subseteq \set{z=re^{i\xi}\in \mathbb{D}: r\geq 1-\nu,~ \abs{\xi-a} \leq \nu}$.  Put $\Phi_n=\{z\in \mathbb{D}: ~2^nh\leq \abs{e^{ia}-z}\leq 2^{n+1}h\}$ and \[ U_1= \int\int_{D_1}K_1(a,h,z)w(\xi)d\xi dr,\quad U_2= \int\int_{D_1^c}K_1(a,h,z)w(\xi)d\xi dr\;.\]

It follows that 
\begin{eqnarray*}
U_1 &\leq& \frac{8h^2}{3} \sum_{n=0}^N \int\int_{\Phi_n}\frac{w(\xi)}{\abs{e^{ia}-z}^3}d\xi dr\\
&\leq & \frac{8h^2}{3} \sum_{n=0}^N \frac{1}{(2^nh)^2}\int_{Q_{2^{n+1}h}(a)} w(\xi)d\xi \\
&\leq & C \sum_{n=0}^N \frac{1}{(2^nh)^2} \int_{Q_{2^{n}h}(a)} w(\xi)d\xi,\quad \mbox{since $w$ is doubling} \\
&\leq & C \sum_{n=0}^N  \int_{2^nh\leq \abs{\xi-a}\leq 2^{n+1}h} \frac{w(\xi)}{(\xi-a)^2}d\xi=C\int_h^{2^{N+1}h} \frac{w(\xi)}{(\xi-a)^2}d\xi \\
&\leq & C \frac{w(J)}{\abs{I}^2} 
\left(\frac{\abs{J}^2}{w(J)}\int_{\xi \notin J} \frac{w(\xi)}{(\xi-a)^2}d\xi \right)<Cw(J), \quad \mbox{since $w \in \mathcal{B}_2$}\\
\end{eqnarray*}
For $z\in D_1^c$, we have 
 \begin{eqnarray*}
 \int\int_{D_1^c}
 \frac{w(\xi)}{\abs{e^{ia}-z}}d\xi dr &\leq  & \sum_{n=0}^{\infty}\int\int_{\Phi_n^*} \frac{w(\xi)}{\abs{e^{ia}-z}}d\xi dr \quad \mbox{using equation \eqref{eqn:Dcomplement}} \\
 &\leq & C \sum_{n=0}^{\infty}\frac{2^n}{h}\int_{1-2^{2-n}h}^{1}\int_{Q_{2^{2-n}h}(a)} w(\xi) d\xi dr \quad \mbox{using again  equation \eqref{eqn:Dcomplement}} \\
  &\leq & C \sum_{n=0}^{\infty}\int_{Q_{2^{1-n}h}(a)} w(\xi) d\xi,  \quad \mbox{since $w$ is a doubling} \\
   &\leq & C \int_{Q_{2h}(a)} w(\xi) d\xi \leq C \int_{Q_{h}(a)} w(\xi) d\xi=Cw(J)\;.
 \end{eqnarray*}
 It follows that for $z\in D_1^c$, 
 \begin{eqnarray*}
U_2 &\leq& Cw(J)\;.
 \end{eqnarray*}
 
 \noindent Part (b): Suppose $w\in \pD_1\cap \pB_2$ such that $w(\xi,r)\equiv \frac{w(1-r)}{1-r}$. Put \[V_1=\int\int_{D_1}K_1(a,h,z)\frac{w(1-r)}{1-r}drd\xi,\quad V_2=\int\int_{D_1^c}K_1(a,h,z)\frac{w(1-r)}{1-r}drd\xi\;.\] Using equation \eqref{eqn:bbK1} above, we have that 
 \begin{eqnarray*}
 V_1 &\leq& \frac{8h^2}{3}\int\int_{D_1} \frac{1}{\abs{e^{ia}-z}^3}\frac{w(1-r)}{1-r}drd\xi\\
&\leq & \frac{8h^2}{3} \sum_{n=0}^N \frac{1}{(2^nh)^3}\int_{-\abs{a}-2^{n+1}h}^{\abs{a}+2^{n+1}h}\int_{1-2^{n+1}h}^1\frac{w(1-r)}{1-r}dr d\xi \\
&\leq & \frac{16h^2}{3} \sum_{n=0}^N \left(\frac{\abs{a}}{(2^nh)^3}+\frac{1}{(2^nh)^2}\right) \int_0^{2^{n+1}h} \frac{w(u)}{u}du,\quad \mbox{by change of variable $u=1-r$}\\
&\leq & \frac{16h^2C}{3} \sum_{n=0}^N  \left(\frac{\abs{a}}{(2^nh)^3}+\frac{1}{(2^nh)^2}\right) w(2^{n}h),\quad \mbox{since $w\in \pD_1$} \\
&=& C(S_{11}+S_{12})
 \end{eqnarray*}
 On one hand 
 \begin{eqnarray*}
S_{11} &=& \sum_{n=0}^N  \frac{1}{(2^nh)^2} w(2^{n}h)\\
&=&  \sum_{n=0}^N \int_{a-2^{n+1}h}^{a+2^{n+1}h}\frac{w(2^{n}h)}{(2^nh)^3}du\leq C \sum_{n=0}^N \int_{a-2^{n+1}h}^{a+2^{n+1}h}\frac{w(u)}{(u)^3}du,\quad \mbox{since $w$ increasing and $u^3<8(2^nh)^3$} \\
&\leq & C \int_{a-h}^{a+h} \frac{w(u)}{u^3}du  \leq C \int_{a-h}^1 \frac{w(u)}{u^3}du\leq \quad \mbox{since $J\subseteq [0,1]$}\\
&\leq & C \frac{w(a-h)}{(a-h)^2}= Chw(a-h),\quad \mbox{since $w\in \pB_2$}\\
&\leq & Cw(J),\quad \mbox{since $w$ is increasing}\;.
 \end{eqnarray*}
 
 On the other hand, we know that $\pD_1\cap \pB_2\subseteq \mathcal{D}=\bigcup\limits_{p>1}\mathcal{B}_p $. Let $p>1$ such that $w\in \mathcal{B}_p$. 
 \begin{eqnarray*}
 S_{12} &=& \sum_{n=0}^N \frac{\abs{a}}{(2^nh)^3} w(2^{n}h)\\
&\leq & \abs{a} (N+1) \left(\sum_{n=0}^N \frac{1}{(2^nh)^3}\right) w(2^{N}h),\quad \mbox{since $w$ is increasing}\\
&\leq & C   \frac{ w(2^{N}h)}{(2^Nh)^{p-1}} \quad \mbox{where $\ds C=\abs{a} (N+1) (2^Nh)^{p-1} \left(\sum_{n=0}^N \frac{1}{(2^nh)^3}\right)$}\\
&\leq & C \int_{2^Nh}^{2^{N+1}h} \frac{ w(2^{N}h)}{(2^Nh)^{p}} du\\
&\leq & 2^pC \int_{2^Nh}^{2^{N+1}h} \frac{ w(u)}{u^p} du\quad \mbox{since $w$ is increasing }\\
&\leq & 2^p C \frac{w(J)}{\abs{J}^p}\left(\frac{\abs{J}^p}{w(J)} \int_{u\notin J} \frac{ w(u)}{u^p} du\right)<Cw(J),\quad \mbox{since $2^{N}h\geq 1$ and $w \in \mathcal{B}_p$.}
 \end{eqnarray*}
Now 
\begin{eqnarray*}
V_2&\leq & \sum_{n=0}^{\infty} \int\int_{\Phi_n^*} \frac{1}{\abs{e^{ia}-z}}\frac{w(1-r)}{1-r}d\xi dr\\
&\leq &  \sum_{n=0}^{\infty} \int_{1-2^{2-n}h}^1\int_{-\abs{a}-2^{2-n}h}^{\abs{a}+2^{2-n}h} \frac{1}{\abs{e^{ia}-z}}\frac{w(1-r)}{1-r}d\xi dr\\
&\leq &  \sum_{n=0}^{\infty} \frac{2^{n-1}}{h}\int_{1-2^{2-n}h}^1\int_{-\abs{a}-2^{2-n}h}^{\abs{a}+2^{2-n}h} \frac{w(1-r)}{1-r}d\xi dr,\quad \mbox{since $\frac{1}{\abs{e^{ia}-z}} < \frac{2^{n-1}}{h}$ on $\Phi_n^*$}\\
&\leq &  \sum_{n=0}^{\infty} \frac{2^{n}}{h} (\abs{a}+2^{2-n}h)\int_0^{2^{2-n}h} \frac{w(u)}{u} du,\quad \mbox{after the change of variable $u=1-r$}\\
&=&S_{21}+S_{22}\;.
\end{eqnarray*}
On one hand, 
\begin{eqnarray*}
S_{21}&=&\sum_{n=0}^{\infty} \frac{2^{n}}{h}2^{2-n}h \int_0^{2^{2-n}h} \frac{w(u)}{u} du =4 \sum_{n=0}^{\infty} \int_0^{2^{2-n}h} \frac{w(u)}{u} du=4  \int_0^{4h} \frac{w(u)}{u} du\\
&\leq & C w(4h)\leq Cw(h)\leq Cw(J), \quad \mbox{since $w\in \pD_1\cap \pB_2\subseteq \mathcal{D}$}\;.\\
\end{eqnarray*}
On the other hand

\begin{eqnarray*}
S_{22}&=& \abs{a} \sum_{n=0}^{\infty} \frac{2^{n}}{h} \int_0^{2^{2-n}h} \frac{w(u)}{u} du\\
&\leq & C \sum_{n=0}^{\infty} \frac{2^{n}}{h}  w(4\cdot 2^{-n}h),\quad \mbox{since $w\in \pD_2$}\\
&\leq & C \sum_{n=0}^{\infty} \frac{2^{n}}{h}  w(2^{-n}h),\quad \mbox{since $w\in \mathcal{D}$}\\
&\leq & C \sum_{n=0}^{\infty} \int_{2^{-n}h}^{2^{1-n}h}  \frac{w(u)}{u}du, \quad \mbox{since $\frac{1}{2^{1-n}h}\leq \frac{1}{u}\leq \frac{1}{2^{-n}h}$} \\
&\leq & C \int_0^{2h} \frac{w(u)}{u}du\leq Cw(2h)\leq Cw(J), \quad \mbox{since $w\in \mathcal{D}$}.
\end{eqnarray*}
\end{proof}

 \begin{proof}[{\bf Proof of Lemma \ref{lemma4}}]\mbox{}\\
 We will start again with the case $d=2$. We know from above that 
  \begin{equation*}
  iK_1(a_1,h_1,z_1)= \frac{2e^{ia_1}(z_1+e^{ia_1})(1-\cos h_1)}{(e^{ia_1}-z_1)((e^{ia_1}-z_1)^2 +2e^{ia_1}z_1(1-\cos h_1))}\;.
  \end{equation*}
  Let us choose $h_1$ and $a_1$ so that $h_1<\abs{e^{ia_1}-z_1}$. Then $\abs{z_1+e^{ia_1}}\geq 2-h_1$.  Since $1-\frac{h_1^2}{2}\leq \cos h_1\leq 1-\frac{h_1^2}{2}+\frac{h_1^4}{24}$, we obtain
  \begin{equation}\label{eqnK1}
  \begin{aligned}
  \abs{K_1(a_1,h_1,z_1)} &\geq  \frac{(2-h_1)\left(h_1^2-\frac{h_1^4}{12}\right)}{\abs{(e^{ia_1}-z_1)}\abs{(e^{ia_1}-z_1)^2 +2e^{ia_1}z_1(1-\cos h_1)}}\\
  &\geq \frac{(2-h_1)\left(h_1^2-\frac{h_1^4}{12}\right)}{\abs{e^{ia_1}-z_1}^3}=\frac{Ch_1^2}{\abs{e^{ia_1}-z_1}^3}\;.
  \end{aligned}
  \end{equation}
  Now let us choose $a_2$ and $h_2$ such that $1<h_2<\abs{e^{ia_2}-z_2}$.
  We also know  from above that 
\[K_2(a_2,h_2,z_2)=\frac{2}{i}\left[\ln(e^{i(a_2-h_2)}-z_2)+\ln(e^{i(a_2+h_2)}-z_2)-2\ln(e^{ia_2}-z_2)\right]\;.\]
We know that $\abs{\ln(e^{ia_2}-z_2)}\geq \ln(\abs{(e^{ia_2}-z_2)})\geq \ln (h_2)>0.$
Therefore 
\begin{equation}\label{eqnK2}
\abs{K_2(a_2,h_2,z_2)}\geq 4\ln(h_2)\;.
\end{equation}
Combining both equations \eqref{eqnK1} and \eqref{eqnK2}, it  follows that for $h_1<\abs{e^{ia_1}-z_1}$ and $1<h_2<\abs{e^{ia_2}-z_2}$ 
\[\abs{f_1(z_1,z_2)}\geq \frac{C(h_1)h_1^2}{\abs{e^{ia_1}-z_1}^3}\;.\]
Now for $d\geq 2$, we can generalize it so that given $1\leq j\leq d$ and $l\neq j$ one has 
\begin{equation}\label{lowerbdfj}
\abs{f_j(\bm{z})}\geq \frac{C(h_j)h_j^2}{\abs{e^{ia_j}-z_j}^3}, \quad \mbox{for}\quad  h_j<\abs{e^{ia_j}-z_j},\quad  1<h_l<\abs{e^{ia_l}-z_l}\;.
\end{equation}
Assume that $z_j=r_ie^{i\xi_j}$ with $r_j<1$.
Then 
\begin{eqnarray*}
\abs{e^{ia_j}-z_j}^2&=& 1-r_j^2-2r_j\cos(\xi_j-a_j)\\
&\leq & (1-r_j)^2+(\xi_j-a_j)^2\;.
\end{eqnarray*}
Let $D_j=\{z_j\in \mathbb{D}: h_j<\abs{e^{ia_j}-z_j} \}$ and $D_j^*=\{z_j\in \mathbb{D}: h_j<\xi_j<a_j+1\}$. Note that  then $D_j\cap D_j^*\neq \emptyset$.   We have that 
\begin{eqnarray*}
\int\int_{\mathbb{D}}\abs{f_j(\bm{z})}w_j(\xi)d\xi_jdr_j&\geq& h_j^2 \int\int_{D_j}\frac{w_j(\xi_i)}{\abs{e^{ia_j}-z_j}^3}d\xi_jdr_j\\
&\geq& h_j^2 \int_0^1\int_{\xi>h_j}\frac{w_j(\xi_i)}{\left((1+r_j)^2+(\xi_j-a_j)^2\right)^{3/2}}d\xi_jdr_j\\
&\geq& h_j^2 \int_{\xi>h_j}\frac{w_j(\xi_i)}{(\xi_j-a_j)^2}d\xi_j \quad \mbox{since $x^{-3}>x^{-2}$ on (0,1)}\\
&\geq& h_j^2 \int_{\xi>h_j}\frac{w_j(\xi_i)}{\xi_j^2}d\xi_j \quad \mbox{since $0<h_j<\xi_j<a_j+1$}\;.\\
\end{eqnarray*}
 \end{proof}
 \begin{proof}[{\bf Proof of Lemma \ref{lemma5}}]
Firstly,  let \[D_{1j}=\set{z_j\in  \mathbb{D}: h_j<\abs{e^{ia_j}-z_j}}.\] 
Therefore, if $1-r_j< \abs{\xi_j-a_j}$, we have $h_j\leq \abs{e^{ia_j}-z_j} \leq \sqrt{2}(\xi_j-a_j)$. There are two possibilities: either $1-r_j<h_j$ or $1-r_j\geq h_j$. So let us consider the subset  \[D_{1j}^*=\set{z_j\in \mathbb{D}: h_j\leq 1-r_j<\abs{\xi_j-a_j}}\] of $D_{1j}$. It follows from equation \eqref{lowerbdfj} that
\begin{eqnarray*}
\int\int_{\mathbb{D}}\abs{f_j(\bm{z})} \frac{w_j(1-r_j)}{1-r_j}d\xi_jdr_j&\geq& \int\int_{D_{2j}}\abs{f_j(\bm{z})} \frac{w_j(1-r_j)}{1-r_j}d\xi_jdr_j \\
&\geq &C(h_j)\int\int_{D_{1j}}\frac{h_j^2}{\abs{e^{ia_j}-z_j}^3} \frac{w_j(1-r_j)}{1-r_j}d\xi_jdr_j\\
&\geq& \frac{h_j^2C(h_j)}{4\sqrt{2}}\int\int_{D_{1j}}\frac{1}{(\xi_j-a_j)^3} \frac{w_j(1-r_j)}{1-r_j}d\xi_jdr_j\\
&=&\frac{h_j^2C(h_j)}{4\sqrt{2}}\int_{0}^{1-h_j} \frac{w_j(1-r_j)}{1-r_j}\rb{\int^{\pi}_{1-r_j+a_j}\frac{1}{(\xi_j-a_j)^3}d\xi_j} dr_j\\
&\geq &C_j\int_{h_j}^{1} \frac{w_j(u_j)}{u_j^3}du_j\;.\\
\end{eqnarray*}
The latter inequality is obtained after the changes of variable $\xi-a=\theta$ and $1-r=u$ respectively,   with $\ds C_j=\frac{h_j^2C(h_j)}{\sqrt{2}}$. 

\noindent Secondly, consider \[D_{2j}=\set{z_j\in  \mathbb{D}: \abs{e^{ia_j}-z_j}\leq \frac{h_j}{4}}.\] 
As above, if $\abs{\xi_j-a_j}<1-r_j$, we have $\abs{e^{ia_j}-z_j}\leq \sqrt{2}(1-r_j)$. So either $\sqrt{2}(1-r_j)<\frac{h_j}{4}$ or $\sqrt{2}(1-r_j)\geq \frac{h_j}{4}$. Thus, consider the subset 
\[D_{2j}^*=\set{z_j\in \mathbb{D}: \abs{\xi_j-a_j}<1-r_j<\frac{h_j}{4\sqrt{2}}}\] of $D_{2j}$. Therefore, we have
\begin{eqnarray*}
\int\int_{\mathbb{D}}\abs{f_j(\bm{z})} \frac{w_j(1-r_j)}{1-r_j}d\xi_jdr_j&\geq&  \int\int_{D_{2j}}\abs{f_j(\bm{z})} \frac{w_j(1-r_j)}{1-r_j}d\xi_jdr_j\\
&\geq &C(h_j)\int\int_{D_{2j}}\frac{h_j^2}{\abs{e^{ia_j}-z_j}^3} \frac{w_j(1-r_j)}{1-r_j}d\xi_jdr_j\\
&=&\frac{C(h_j)}{16\sqrt{2}}\int_{0}^{1-\frac{h_j}{4\sqrt{2}}} \frac{w_j(1-r_j)}{(1-r_j)^2}\rb{\int_{a_j+r_j-1}^{a_j+1-r_j}d\xi_j} dr_j\\
&=&C_j\int_{0}^{\frac{h_j}{4\sqrt{2}}} \frac{w_j(u_j)}{u_j}du_j,\\
\end{eqnarray*}
after the change of variable $u=1-r$ with $\ds C_j=\frac{C(h_j)}{16\sqrt{2}}\;.$
 \end{proof}
\section{Conclusion} \label{conclusion}
We have proposed in this paper a space that acts as the analytic extension of the so-called special atom spaces in higher dimensions. What we find interesting and remarkable in these spaces is their apparent simplicity. Certainly one could think of atoms defined on intervals that are different from the characteristic functions on these intervals, say for instance polynomial type of atoms. However, the properties like orthonormality would be difficult to prove, especially in higher dimensions. The results in this paper open the door to exploring lacunary sequences in Bergman-Besov-Lipschitz spaces in higher dimensions. Another idea that is no more far-fetched in the idea of Blaschke-products in these spaces, that are non-complex in their inception but have complex extensions. 
\section{Data Availability}
Data sharing not applicable to this article as no datasets were generated or analyzed during the current study.
\bibliography{MyBib222}

\begin{thebibliography}{10}

\bibitem{Arcozzi2006}
N.~Arcozzi, R.~Rochberg, and E.~Sawyer.
\newblock Carleson measures and interpolating sequences for besov spaces on
  complex balls.
\newblock {\em Mem. Amer. Math. Soc.}, 186~(559), 2006.
\newblock vi+163.

\bibitem{Beurling1948}
A.~Beurling.
\newblock On two problems concerning linear transformations in hilbert space.
\newblock {\em Acta Math.}, 81~(17), 1948.

\bibitem{DeSouza1989}
S.~Bloom and G.~De~Souza.
\newblock Atomic decomposition of generalized lipschitz spaces.
\newblock {\em Illinois Journal of Mathematics}, pages 682--686, 1989.

\bibitem{BlandigneresEtAl}
A.~Brandigneres, E.~Fricain, F.~Gaunard, A.~Hartman, and W.~T. Ross.
\newblock Direct and reverse carleson measures for $\h(b)$ spaces.
\newblock {\em Indiana Univ. Math. J.}, 64~(4):1027--1057, 2015.

\bibitem{Coifman1974}
R.~Coifman.
\newblock A real variable characterization of $h^p$.
\newblock {\em Studia Mathematica}, 51~:269--274, 1974.

\bibitem{DeSouza1985}
G.~De~Souza.
\newblock The atomic decomposition of bergman-besov-lipschitz spaces.
\newblock {\em Proc. Amer. Math. Soc.}, 14~:682--686, 1985.

\bibitem{DeSouza1994}
G.~De~Souza and G.~Sampson.
\newblock function in the dirichlet space such that its fourier series
  divergesalmost everywhere.
\newblock {\em Proceedings of the American Mathematical Society},
  120~(3):723--726, 1994.

\bibitem{Fefferman1971}
C.~Fefferman.
\newblock Characterization of bounded mean oscillation.
\newblock {\em Bulletin of the American Mathematical Society}, 77~(4):587--588,
  1971.

\bibitem{Folland1999}
G.~G. Folland.
\newblock {\em Real analysis, modern techniques abd their applications}.
\newblock Wiley-Interscience, 1999.

\bibitem{FricainMashreghi}
E.~Fricain and J.~Mashreghi.
\newblock {\em New Mathematical Monographs}, volume~1~ of {\em xix+681 pp.}
\newblock Cambridge University Press, 2016.

\bibitem{IdrissiElFallah}
B.-E. Idrissi and O.~El~Farah.
\newblock Blaschke products and zero sets in weighted dirichlet spaces.
\newblock {\em Potential Analysis}, 2020.

\bibitem{Kwessi2013}
E.~Kwessi, G.~De~Souza, A.~Abebe, and R.~Aulaskari.
\newblock Characterization of lacunary functions in bergman-besov-lipschitz
  spaces.
\newblock {\em Complex Variables and Elliptics Equations}, 58~(2):157--162,
  2013.

\bibitem{Kwessi2019}
E.~Kwessi, G.~De~Souza, D.~Ngalla, and N.~Mariama.
\newblock The special atom space in higher dimensions.
\newblock {\em Demonstratio Mathematica}, 53~:131--151, 2020.

\bibitem{Muckenhoupt1972}
B.~Muckenhoupt.
\newblock Weighted inequalities for the hardy maxima operator.
\newblock {\em Transactions of the American Mathematical Society},
  165~:107--226, 1972.

\bibitem{Rudin1987}
W.~Rudin.
\newblock {\em Real and Complex Analysis}.
\newblock McGraw-Hill, 1987.

\bibitem{KeheZu}
K.~Zhu.
\newblock Bloch spaces of analytic functions.
\newblock {\em Rocky Mountain Journal of Mathematics}, 23~(3):1143--1177, 1993.

\end{thebibliography}

\end{document}